\documentclass[a4paper, 12pt]{article}
\usepackage{amsthm,amsmath,amssymb,vmargin}
\usepackage[latin1]{inputenc}
\usepackage{tikz}
\tikzstyle{empty}=[circle,draw=black!80,thick]
\tikzstyle{nero}=[circle,draw=black!80,fill=black!80,thick]
\usepackage{comment}
\usepackage{caption}
\title{\textbf{On permutation modules and decomposition numbers for symmetric groups}}                    
\author{Eugenio Giannelli}
\date{}
\newtheorem{teo1.2}{Theorem 1.2}[section]
\newtheorem{teo1.1}{Theorem 1.1}[section]
\newtheorem{teo}{Theorem}[section]
\newtheorem{cor}[teo]{Corollary}
\newtheorem{lemma}[teo]{Lemma}

\newtheorem{prop}[teo]{Proposition}
\newtheorem{esem}[teo]{Example}

\newtheorem*{conjecture}{Conjecture}

\theoremstyle{definition}

\begin{document}
\bibliographystyle{archmath}

\maketitle

%\tableofcontents

\begin{abstract}
We study the indecomposable summands of the permutation module obtained by inducing the trivial $\mathbb{F}(S_a\wr S_n)$-module to the full symmetric group $S_{an}$ for any field $\mathbb{F}$ of odd prime characteristic $p$ such that $a<p\leq n$.
In particular we characterize the vertices of such indecomposable summands. As a corollary we will disprove a modular version of Foulkes' Conjecture.

In the second part of the article we will use this information to give a new description of some columns of the decomposition matrices of symmetric groups
in terms of the ordinary character of the Foulkes module $\phi^{(a^n)}$.
\end{abstract}

\section{Introduction}
The determination of the decomposition matrices and the study of the modular structure of permutation modules are two important open problems in the representation theory of symmetric groups.
 
Young permutation modules were deeply studied by James in \cite{JamesErd}, Klyachko in \cite{KLY} and Grabmeier in \cite{Grab}. They completely parametrized 
the indecomposable summands of such modules (known as Young modules) and developed a Green correspondence for those summands.
Their original description of the modular structure of Young modules was based on Schur algebras. 
More recently Erdmann in \cite{KE} described completely the Young modules using only the representation theory of the symmetric groups. 
In particular it is proved that the vertex of a given Young module of the symmetric group $S_n$ of degree $n$ is always conjugate to a Sylow $p$-subgroup of $S_{\lambda_1}\times S_{\lambda_2}\times\cdots\times S_{\lambda_k}$, where $(\lambda_1,\ldots,\lambda_k)$ is a partition of $n$.

The problem of finding decomposition numbers for symmetric groups in prime characteristic has been studied extensively.
The decomposition
matrix of the symmetric group~$S_n$ in prime characteristic~$p$
has rows labelled by the partitions of $n$, and columns
by the so called $p$-regular partitions of $n$, namely, partitions of $n$ with less than $p$ parts of any given size.
The decomposition number $d_{\mu\nu}$ is the entry of the decomposition matrix that 
records the number of composition factors of the Specht module $S^\mu$,
defined over a field of characteristic~$p$, 
that are isomorphic
to the simple module $D^\nu$, defined by James in \cite{JamesIrrepsBPLMS} as the unique top composition
factor of $S^\nu$. 

The main purpose of this paper is to begin the study, over a field of prime characteristic, of a new family of permutation modules known as Foulkes modules.
The new information obtained will allows us to draw corollaries on decomposition numbers. In particular we will give 
a new combinatorial description of certain columns of the decomposition matrices of symmetric groups.

 In characteristic zero the study of Foulkes modules was mainly connected to the long standing Foulkes' Conjecture, stated by H.O. Foulkes in \cite{Foulkes}. For some recent advances on the ordinary character of the Foulkes modules see \cite{Giannelli} and \cite{PagetWildon}.
For any $a$ and $n$ natural numbers the Foulkes module $H^{(a^n)}$ is  obtained by inducing the trivial representation of the wreath product $S_a\wr S_n$ up to the full symmetric group $S_{an}$.  If the prime $p$ is the characteristic of the underlying field $\mathbb{F}$, we observe that whenever $n<p$ then the indecomposable summands of the Foulkes module are Young modules, since $H^{(a^n)}$ is isomorphic to a direct summand of the Young permutation module $M^{(a^n)}$ in this case.
On the other hand for $n\geq p$ we do not have any precise information on the non-projective indecomposable summands of $H^{(a^n)}$.
We will focus on the study of these new indecomposable $\mathbb{F}S_{an}$-modules. 
%where $\mathbb{F}$ is the finite field of size $p$. 
In particular we will give the following description of the possible vertices for all $a$ and $n$ such that $a<p\leq n$. 
\begin{teo}\label{T1}
Let $p$ be an odd prime and let $a$ and $n$ be natural numbers such that $a<p\leq n$.
Let $U$ be an indecomposable non-projective summand of the $\mathbb{F}S_{an}$-module $H^{(a^n)}$ and let $Q$ be a vertex of $U$. Then there exists $s\in\{1,2,\ldots,\lfloor\frac{n}{p}\rfloor\}\cap\mathbb{N}$ such that $Q$ is conjugate to a Sylow $p$-subgroup of $S_a\wr S_{sp}$.
Moreover  
the Green correspondent of $U$ admits a tensor factorization
$V \boxtimes Z$ as a module for $\mathbb{F}((N_{S_{asp}}(Q)/Q) \times S_{a(n-sp)})$, where
$V$ is isomorphic to the projective cover of the trivial $\mathbb{F}(N_{S_{asp}}(Q)/Q)$-module and $Z$ is  an
indecomposable projective summand of $H^{(a^{n-sp})}$.
\end{teo}
This result should be compared with \cite[Theorem 1.2]{GW} which studies summands of certain twists by the $\mathrm{sign}$ character of the permutation module $H^{(2^n)}$. The main tool used to study the vertices of $H^{(a^n)}$ is the Brauer correspondence for $p$-permutation modules as developed by Brou\'e in \cite{Broueperm}.

 It is clear by Theorem \ref{T1} that the non-projective indecomposable summands of the permutation module $H^{(a^n)}$ are not Young modules for any $a<p\leq n$, since their vertices are not Sylow $p$-subgroups of Young subgroups of $S_{an}$. This observation will be sufficient to disprove the modular version of Foulkes' Conjecture. 

In order to present our new result on decomposition numbers we need to introduce the following definition. 
Let $\gamma$ be a $p$-core partition and let $\phi^{(a^n)}$ be the ordinary character afforded by the Foulkes module $H^{(a^n)}$. 
Denote by $\mathcal{F}(\gamma)$ the set containing all the partitions $\mu$ of $an$ such that the $p$-core of $\mu$ is equal to $\gamma$ and such that the irreducible ordinary character of $S_{an}$ labelled by $\mu$ has non-zero multiplicity in the decomposition of $\phi^{(a^n)}$ as a sum of irreducible characters. In symbols $$\mathcal{F}(\gamma)=\{\mu\vdash an\ :\ \gamma(\mu)=\gamma\ \ \text{and}\ \ \left\langle\chi^\mu, \phi^{(a^n)}\right\rangle\neq 0\}.$$

The new results obtained in Theorem \ref{T1} will lead us to prove the following theorem. 

\begin{teo}\label{TTT1}
Let $a,n$ be natural numbers and let $p$ be a prime such that $a<p\leq n$.
Let $\lambda$ be a $p$-regular partition of $na$ such that $\lambda$ has weight $w<a$. Denote by $\gamma$ the $p$-core of $\lambda$.
If $\lambda$ is maximal in $\mathcal{F}(\gamma)$,
then the only non-zero entries in the column labelled by $\lambda$ of the decomposition matrix of $S_{an}$ are in the rows labelled by partitions $\mu\in\mathcal{F}(\gamma)$. Moreover 
$$[S^\mu:D^\lambda]\leq \left\langle\phi^{(a^n)},\chi^{\mu}\right\rangle.$$
\end{teo}

Theorem \ref{TTT1} allows us to detect new information on decomposition numbers from the study of the ordinary structure of the Foulkes character. 
In particular the study of the zero-multiplicity characters in the decomposition of $\phi^{(a^n)}$ leads to some new non-obvious zeros in certain columns of the decomposition matrix of $S_{an}$ (see Corollary \ref{x} and Example \ref{E:last}).

\vspace{.5cm}

The paper is structured as follows.
 The necessary background on the Brauer correspondence and a detailed description of the combinatorial structure of the Foulkes modules are given in Section $2$.
In Section $3$ we will prove Theorem \ref{T1}. The proof is split into a series of lemmas and propositions describing the structural properties of the Foulkes modules.
Finally in Section $4$ we will prove Theorem \ref{TTT1} and give some applications.

%In the second part of the paper we will study the Twisted module $T^{(a^b)}$. Closely related to the Foulkes module, the Twisted module will have similar but different possible vertices. In particular we will prove that all the indecomposable summands of $T^{(2^n)}$ have vertex a Sylow $p$-subgroup of $S_2\wr S_n$.  

\section{Preliminaries}

In this section we briefly present the theoretical tools that we will extensively use to deduce our new results. A partition $\lambda$ of a natural number $n$ is a weakly decreasing sequence of positive numbers $\lambda=(\lambda_1,\lambda_2,\ldots,\lambda_k)$ such that $\sum_{i=1}^k\lambda_i=n$.
%Given $\lambda=(\lambda_1,\ldots,\lambda_k)$ and $\mu=(\mu_1,\ldots,\mu_s)$ partitions of $n$ and $m$ respectively with $k\geq s$, we denote by $\lambda+\mu$ the partition $$(\lambda_1+\mu_1,\lambda_2+\mu_2,\ldots,\lambda_s+\mu_s,\lambda_{s+1},\ldots,\lambda_k).$$
Given a natural number $n$ and a partition $\lambda$ of $n$ we will denote by $S^{\lambda}$ the Specht module labelled by $\lambda$ and by $\chi^\lambda$ the ordinary character afforded by
$S^\lambda$. This notation agrees with the notation of \cite{James} and we refer the reader to it for a comprehensive exposition of the general theory of the symmetric group representations.

\subsection{Brou\'e's correspondence}

As mentioned in the introduction we will study, in the next section, the vertices of a family of $p$-permutation modules. 
One of the most important techniques used will be the Brauer construction applied to $p$-permutation modules. 
Here we will summarize the main results of Brou\'e's paper \cite{Broueperm}. 

Let $\mathbb{F}$ be a field of prime characteristic $p$ and $G$ a finite group. An $\mathbb{F}G$-module $V$ is called a $p$-permutation module if for every $P\in \mathrm{Syl}_p(G)$ there exists a linear basis $\mathcal{B}_P$ of $V$ that is permuted by $P$. It is quite easy to derive from the definition that the direct sum and the tensor product of $p$-permutation modules are $p$-permutation modules; if $H\leq G$ then the restriction to $H$ of a $p$-permutation $\mathbb{F}G$-module is a $p$-permutation module as well as the induction from $H$ to $G$ of a $p$-permutation $\mathbb{F}H$-module; in conclusion every direct summand of a $p$-permutation module is a $p$-permutation module.
A complete characterization of such modules is contained in the following theorem.

\begin{teo}
An indecomposable $\mathbb{F}G$-module $V$ is a $p$-permutation module if and only if there exists $P\leq_p G$ such that $V\ |\ \mathbb{F}\uparrow_P^G$.
\end{teo}

We recall now the definition and the first basic properties of the Brauer construction for $\mathbb{F}G$-modules. 
Given an $\mathbb{F}G$-module $V$ and $Q\leq_p G$ we denote by $V^Q$ the set of fixed elements $\{v\in V\ |\ vg=v\ \text{for all}\ g\in Q \}$. It is easy to see that $V^Q$ is an 
$\mathbb{F}N_G(Q)$-module on which $Q$ acts trivially. For $P$ a proper subgroup of $Q$, the relative trace map $\mathrm{Tr}_P^Q:V^P\rightarrow V^Q$ is the linear map defined by $$\mathrm{Tr}_P^Q(v)=\sum_{g\in Z}vg,$$
where $Z$ is a set of right coset representatives for $P$ in $Q$. It is easy to notice that the definition of the map does not depend on the choice of the set of representatives. We observe that also $$\mathrm{Tr}^Q(V):=\sum_{P<Q}\mathrm{Tr}_P^Q(V^P)$$ is an $\mathbb{F}N_G(Q)$-module on which $Q$ acts trivially. Therefore we can define the $\mathbb{F}(N_G(Q)/Q)$-module called the \textit{Brauer quotient} of $V$ with respect to $Q$ by $$V(Q)=V^Q/\mathrm{Tr}^Q(V).$$
For our scope it is very important to remark that if $V$ is an indecomposable $\mathbb{F}G$-module and $Q\leq_p G$ then $V(Q)\neq 0$ implies that $Q$ is contained in a vertex of $V$. Brou\'e proved in \cite{Broueperm} that the converse holds in the case of $p$-permutation modules.

\begin{teo}\label{BT2}
Let $V$ be an indecomposable $p$-permutation module and $Q$ be a vertex of $V$. Let $P$ be a $p$-subgroup of $G$, then $V(P)\neq 0$ if and only if $P\leq Q^g$ for some $g\in G$.
\end{teo}

If $V$ is an $\mathbb{F}G$-module with $p$-permutation basis $\mathcal{B}$ with respect 
to a Sylow $p$-subgroup $P$ of $G$ and $R \le P$, then 
taking for each orbit of $R$ on $\mathcal{B}$ the sum
of the vectors in that orbit, we obtain a basis for $V^R$.
The sums over vectors lying in orbits of size $p$ or more
are relative traces from proper subgroups of~$R$, and so $V(R)$ is isomorphic to the $\mathbb{F}$-span of
$$\mathcal{B}(R):= \{ v \in \mathcal{B} : \text{$vg = v$ for all $g\in R$} \}.$$
Thus Theorem~\ref{BT2} has the following
corollary,

\begin{cor}\label{cor:Brauer}
Let $V$ be a $p$-permutation $\mathbb{F}G$-module
with 
$p$-permutation basis
$\mathcal{B}$ with respect to a Sylow $p$-subgroup $P$ of $G$.
Let $R \leq P$. Then the $\mathbb{F}N_G(R)$-module $V(R)$ is equal to %the $F$-span of 
$\langle \mathcal{B}(R) \rangle$ %. %_F$
and $V$ has an indecomposable summand with a
vertex containing~$R$ if and only if $\mathcal{B}(R) \not= \varnothing$.
\end{cor}

The next result \cite[3.4]{Broueperm} explains what is now known as Brou\'e's correspondence.

\begin{teo}\label{BC1}
An indecomposable $p$-permutation module $V$ has vertex $Q$ if and only if $V(Q)$ is a projective $\mathbb{F}(N_G(Q)/Q)$-module. Furthermore
\begin{itemize}
\item The Brauer map sending $V$ to $V(Q)$ is a
bijection between the set of indecomposable $p$-permutation 
$\mathbb{F}G$-modules with vertex $Q$ and the set of indecomposable 
projective $\mathbb{F}(N_G(Q)/Q)$-modules. Regarded as an $\mathbb{F}N_{G}(Q)$-module, $V(Q)$ is the
Green correspondent of $V$.
\item Let $V$ be a $p$-permutation $\mathbb{F}G$-module and $E$ an indecomposable projective $\mathbb{F}(N_G(Q)/Q)$-module. Then $E$ is a direct summand of $V(Q)$ if and only if its correspondent $U$ $(\text{i.e the $\mathbb{F}G$-module $U$ such that}\ U(Q)\cong E)$ is a direct summand of $V$.
\end{itemize}
\end{teo}

Some important consequences that we will use extensively in the next sections are stated below.

\begin{cor}
Let $U$ be a $p$-permutation $\mathbb{F}G$-module and let $Q$ be a $p$-subgroup of $G$. The Brauer correspondent $U(Q)$ of $U$ is a $p$-permutation
$\mathbb{F}N_G(Q)$-module.
\end{cor}
\begin{proof}
Let $P'$ be a Sylow $p$-subgroup of $N_G(Q)$ and let $P$ be a Sylow $p$-subgroup of $G$ containing $P'$. 
Denote by $\mathcal{B}_P$ a $p$-permutation basis of $U$ with respect to $P$. By Corollary \ref{cor:Brauer} we have that the $\mathbb{F}N_G(Q)$-module
$U(Q)$ has linear basis $\mathcal{B}_P(Q)$. It is easy to observe that $\mathcal{B}_P(Q)$ is a $p$-permutation basis with respect to $P'$.
Let $P''$ be another Sylow $p$-subgroup of $N_{G}(Q)$ and let $g\in N_G(Q)$ such that $P''=P'^g$. 
Then we have that $$\mathcal{B}'':=\{xg\ |\ x\in\mathcal{B}_P(Q) \}$$ is a $p$-permutation basis of $U(Q)$ with respect to $P''$.
This completes the proof.
\end{proof}

\begin{cor}\label{BC3}
Let $G$ and $H$ be two finite groups and let $C$ be a subgroup of $G$.
Let $U$ be an indecomposable $p$-permutation $\mathbb{F}G$-module and $V$ be a $p$-permutation $\mathbb{F}H$-module. Then 
\begin{itemize}
\item If $U\downarrow_C=W_1\oplus\cdots\oplus W_k$, then there exist a vertex $R$ of $U$ and vertices $Q_1,\ldots,Q_k$ of $W_1,\ldots,W_k$ respectively, such that $Q_i\leq R$ for all $i\in\{1,2,\ldots,k\}$.
\item The indecomposable $\mathbb{F}(G\times H)$-module $U\boxtimes V$ has a vertex containing both $Q$ and $P$, vertices of $U$ and $V$ respectively. 
\end{itemize}
\end{cor}

\begin{lemma}\label{Sum}
Let $U$ and $V$ be $p$-permutation $\mathbb{F}G$-modules and let $P$ be a $p$-subgroup of $G$, then $$(U\oplus V)(P)\cong U(P)\oplus V(P)$$ as 
$\mathbb{F}_p(N_G(P))$-modules.
\end{lemma}

The following lemma is stated in \cite[Lemma 4.7]{GW}; we include the proof here to make this article more self contained.

\begin{lemma}\label{LDB}
Let $Q$ and $R$ be $p$-subgroups of a finite group $G$ and let 
$U$ be a $p$-permutation $\mathbb{F}G$-module. Let $K = N_G(R)$.
If $R$ is normal in $Q$ then the Brauer correspondents $U(Q)$ and 
$\bigl( U(R) \bigr)(Q)$ are isomorphic as $\mathbb{F}N_{K}(Q)$-modules.
\end{lemma}

\begin{proof}
Let $P$ be a Sylow $p$-subgroup of $N_G(R)$ containing $Q$
and let  $\mathcal{B}$ be a $p$-permutation basis for $U$ with respect to $P$.
By Corollary~\ref{cor:Brauer} we have $U(Q) = \langle \mathcal{B}(Q) \rangle$ as an 
$\mathbb{F}N_G(Q)$-module. 
In particular 
$$U(Q)\big\downarrow_{N_K(Q)}\! {}={} \langle \mathcal{B}(Q) \rangle$$
as an $\mathbb{F}N_{K}(Q)$-module.
On the other hand $U(R) = \langle \mathcal{B}(R) \rangle$ as an $\mathbb{F}N_G(R)$-module. Now $\mathcal{B}(R)$ is a
$p$-permutation basis for~$U(R)$ with respect to $K \cap P$.
Since this subgroup contains $Q$ we have $\bigl( U(R) \bigr) (Q) =
\langle \mathcal{B}(R) \rangle (Q) = \langle (\mathcal{B}(R))(Q) \rangle
= \langle \mathcal{B}(Q) \rangle$, as $\mathbb{F}N_{K}(Q)$-modules, as required.
\end{proof}

\begin{lemma}\label{LDB2}
Let $G$ and $G'$ be finite groups and let $U$ and $U'$ be $p$-permutation
modules for $\mathbb{F}G$ and $\mathbb{F}G'$, respectively. 
If $Q \le G$ is a $p$-subgroup then $$(U \boxtimes U')(Q) = U(Q) \boxtimes U',$$
where on the left-hand side
$Q$ is regarded as a subgroup of $G \times G'$ in the obvious way.
\end{lemma}

\begin{proof}
This follows easily from Corollary~\ref{cor:Brauer} by taking
$p$-permutation bases for $U$ and $U'$.
\end{proof}

\begin{prop}\label{pp1}
Let $M$ be a $p$-permutation $\mathbb{F}G$-module and let $P$ be a $p$-subgroup of $G$. If $M(P)$ is an indecomposable $\mathbb{F}(N_G(P))$-module then
$M$ has a unique indecomposable summand $U$ such that $P$ is contained in a vertex of $U$.
\end{prop}

\begin{proof}
Suppose by contradiction that exist $V_1$ and $V_2$ indecomposable summands of $M$ with vertices $Q_1$ and $Q_2$ respectively, such that $P\leq Q_1,Q_2$. 
Then by Lemma~\ref{Sum} we have that $$V_1(P)\oplus V_2(P)\ |\ M(P).$$
This contradicts the indecomposability of $M(P)$ since by Theorem \ref{BC1} we have that $V_i(P)\neq 0$ for $i\in\{1,2\}$.
\end{proof}

An argument that we shall use several times is stated in the lemma below:

\begin{lemma}\label{Lemma:perm} If $P$ is a $p$-group and $Q$ is a subgroup of $P$ then
the permutation module $\mathbb{F}\uparrow_Q^P$ is indecomposable, with vertex $Q$. $\qed$
\end{lemma}

We conclude the section by recalling the definition and the basic properties of Scott modules. We refer the reader to \cite[Section 2]{Broueperm} for a more detailed account. 
Given a subgroup $H$ of $G$ there exists a unique indecomposable summand $U$ of the permutation module $\mathbb{F}\uparrow_H^G$ such that 
the trivial $\mathbb{F}G$-module is a submodule of $U$. We say that $U$ is the \textit{Scott module} of $G$ associated to $H$ and we denote it by 
$\mathrm{Sc}(G,H)$. The following theorem summarizes the main properties of Scott modules.
\begin{teo}\label{scott}
Let $G$ be a finite group, $H$ a subgroup of $G$ and $P$ a Sylow $p$-subgroup of $H$. Then the Scott module $\mathrm{Sc}(G,P)$ is isomorphic to $\mathrm{Sc}(G,H)$ and is uniquely determined 
up to isomorphism by either of the following properties:
\begin{itemize}
\item  The trivial $\mathbb{F}G$-module is isomorphic to a submodule of $\mathrm{Sc}(G,P)$.
\item  The trivial $\mathbb{F}G$-module is isomorphic to a quotient of $\mathrm{Sc}(G,P)$.
\end{itemize}
Moreover, $\mathrm{Sc}(G,P)$ has vertex $P$ and the Brou\'e correspondent  $(\mathrm{Sc}(G,P))(P)$ is isomorphic to the projective 
cover of the trivial $\mathbb{F}(N_{G}(P)/P)$-module
\end{teo}

\subsection{Blocks of symmetric groups}

The blocks of symmetric groups are combinatorially described by Nakayama's Conjecture, first proved by Brauer and Robinson in \cite{BrauerNakayama} and \cite{RobinsonNakayama}. In order to state this result, we must 
recall some definitions. 

Let $\lambda$ be a partition. 
A \emph{$p$-hook} in $\lambda$ is a connected part of the 
rim of the Young diagram of $\lambda$ consisting of exactly $p$ boxes,
whose removal leaves the diagram of a partition. 
By repeatedly stripping off $p$-hooks from~$\lambda$
we obtain the \emph{$p$-core} of $\lambda$; the number
of hooks we remove is the \emph{weight} of $\lambda$.

\begin{teo}[Nakayama's Conjecture]\label{thm:Nakayama}
Let $p$ be prime.
The $p$-blocks of  $S_n$ are labelled by pairs $(\gamma, w)$, where
$\gamma$ is a $p$-core and $w \in \mathbf{N}_0$ is the associated weight,
such that $\left|\gamma\right| + wp = n$.
Thus the Specht module $S^\lambda$ lies
in the block labelled by $(\gamma, w)$ if and only if $\lambda$
has $p$-core~$\gamma$ and weight $w$.\hfill$\Box$
\end{teo}

Often it is best to consider partitions using James' \textit{abacus}: for a detailed account of this tool see \cite[pages 76-78]{JK}.  
For example given a partition $\lambda$, it is very easy to understand its $p$-core $\gamma(\lambda)$ by using the abacus.

The following result on the defect group of Blocks of the symmetric group (see \cite[Theorem 6.2.45]{JK}) will be important in the proof of our main Theorem \ref{TTT1}.

\begin{teo}\label{Thm:SymDefectGroups}
The defect group of a symmetric group block of weight $w$ is conjugate to a Sylow
$p$-subgroup of $S_{wp}$.
\end{teo}

\subsection{The Foulkes module}

Let $\mathbb{K}$ be a field and $a$, $n$ two non-zero natural numbers. Let $\Omega^{(a^n)}$ be the collection of all set partitions of $\{1,2,\ldots,an\}$ into $n$ sets of size $a$. We will denote an arbitrary element $\omega\in \Omega^{(a^n)}$ by $$\omega=\{\omega_1,\omega_2,\cdots,\omega_n\},$$ 
where $\omega_j\subseteq\{1,2,\cdots,an\}$, $|\omega_j|=a$ and $\omega_i\cap \omega_j=\emptyset$ for all $1\leq i<j\leq n$. We will call
$\omega_j$ a \emph{set of $\omega$}.
The symmetric group $S_{an}$ acts transitively in a natural way on $\Omega^{(a^n)}$ by permuting the numbers in each set of every set partition.
Let $H^{(a^n)}$ be the $\mathbb{K}S_{an}$-permutation module generated as a $\mathbb{K}$-vector space by the elements of $\Omega^{(a^n)}$, with the action of $S_{an}$ defined as the natural linear extension of the action on $\Omega^{(a^n)}$. The module $H^{(a^n)}$ is called a \textit{Foulkes module}. 
Since the action of $S_{an}$ is transitive on $\Omega^{(a^n)}$ and the stabilizer of any set partition $\omega\in\Omega^{(a^n)}$ is isomorphic to 
$S_a\wr S_n$ it is finally easy to deduce that $$H^{(a^n)}\cong \mathbb{K}\big\uparrow_{S_a\wr S_n}^{S_{an}}.$$

When the field $\mathbb{K}$ is the field of complex numbers $\mathbb{C}$, the study of the decomposition into irreducible 
direct summands of the Foulkes module is closely related to the problem known as Foulkes' Conjecture as stated firstly in \cite{Foulkes} by H.O. Foulkes in $1950$.

\begin{conjecture}
Let $\mathbb{K}=\mathbb{C}$ and let $a$ and $n$ be natural numbers such that $a<n$. Then $H^{(n^a)}$ is a direct summand of $H^{(a^n)}$.
\end{conjecture}

If we replace $\mathbb{C}$ with $\mathbb{F}$ (a field of prime characteristic $p$) in the statement above we obtain a modular version of Foulkes' Conjecture. 
This version is known to be false but we are not aware of any explicit reference in the literature. We will give a short proof of this fact as a corollary of Theorem \ref{T1}.

\section{On the indecomposable summands of $H^{(a^n)}$}

This section is devoted to the proof of Theorem \ref{T1}. As explained in the introduction, the proof is split into a series of lemmas and propositions.
The structure of the section is similar to Section 4 of \cite{GW} but different ideas and further \textit{ad hoc} arguments are needed here.
We start by fixing some notation. Let $\mathbb{F}$ be a field of odd prime characteristic $p$ and let $a$ and $n$ be natural numbers such that $a<p\leq n$. Let $S_a\wr S_n$ be the subgroup of $S_{an}$ acting transitively and imprimitively on $\{1,2,\ldots,an\}$ and having as blocks of imprimitivity the sets 
$$T_j=\{j,n+j,2n+j,\ldots,(a-1)n+j\}$$ for $j\in\{1,\ldots,n\}$. 
In this setting we have that for any Sylow $p$-subgroup $P$ of $S_a\wr S_n$ there exists a Sylow $p$-subgroup $Q$ of $S_{\{1,\ldots,n\}}$, such that $P$ is conjugate to $$\overline{Q}=\{\overline{x}\ |\ x\in Q\},$$
where $(j+kn)\overline{x}=(j)x+kn$ for all $j\in\{1,\ldots,n\}$ and all $k\in\{0,1,\ldots,a-1\}$.
Let $\rho$ be an element of order $p$ in $S_a\wr S_n$. By the above remarks there exists $s\in\{1,2,\ldots,\lfloor\frac{n}{p}\rfloor\}\cap\mathbb{N}$ such that $\rho$ has $sa$ orbits of order $p$ and $a(n-sp)$ fixed points in its natural action on $\{1,2,\ldots,an\}$. 

For all $j\in\mathbb{N}$ such that $pj\leq an$ let $z_j$ be the $p$-cycle of $S_{an}$ defined by $$z_j=(p(j-1)+1,p(j-1)+2,\ldots,pj).$$ Denote by $R_\ell$ the cyclic subgroup of $S_{an}$ of order $p$ generated by $z_1z_2\cdots z_\ell$. %for any $\ell\in\mathbb{N}$ such that $\ell p\leq an$.
We will call $\mathcal{O}_1,\ldots,\mathcal{O}_\ell$ the \textit{$p$-orbits} of $R_\ell$. Note that $\mathcal{O}_j=\mathrm{supp}(z_j)$ for all $j\in\{1,2,\ldots,\ell\}$. 

In the following lemma we will study the Brou\'e correspondence for $H^{(a^n)}$ with respect to $R_\ell$.

\begin{lemma}\label{L1}
Let $a$ and $n$ be natural numbers and $p$ an odd prime such that $a<p\leq n$. Let $\ell$ be a natural number such that $\ell p\leq an$. If $\ell=as$ for some natural number $s$, then $$H^{(a^n)}(R_{as})\cong H^{(a^{sp})}(R_{as})\boxtimes H^{(a^{n-sp})}$$ 
as $\mathbb{F}N_{S_{an}}(R_{as})$-modules. If $\ell$ is not an integer multiple of $a$ then $H^{(a^n)}(R_\ell)=0$.
\end{lemma}

\begin{proof}
We already noticed that the number of $p$-orbits of an element of order $p$ in $S_a\wr S_n$ must be a multiple of $a$. Therefore if $\ell$ is not an integer multiple of $a$ then $R_\ell$ is not conjugate to any subgroup of $S_a\wr S_n$. This implies that $H^{(a^n)}(R_\ell)=0$ by Theorem \ref{BT2}.

Suppose now that $\ell=as$ for some $s\in\mathbb{N}$. Let $\omega=\{\omega_1,\omega_2,\ldots,\omega_n\}\in\Omega^{(a^n)}$ be fixed by $R_{as}$. Then there exist $\omega_{j_1},\ldots,\omega_{j_{sp}}$ sets of $\omega$ such that $$\bigcup_{i=1}^{sp}\omega_{j_i}=\mathrm{supp}(R_{as})$$
since no set of a fixed set partition can contain two numbers $x$ and $y$ such that $x\in \mathrm{supp}(R_{as})$ and $y\notin \mathrm{supp}(R_{as})$. 
So we can write each $\omega\in\Omega^{(a^n)}(R_{as})$ as $\omega=u_\omega\cup v_\omega$ where $$u_\omega=\big\{\omega_{j_1},\ldots,\omega_{j_{sp}}\big\}\in\Omega^{(a^{sp})}(R_{as}),$$ and $v_\omega$ is a set partition in $\Omega^{(a^{n-sp})}_+$, that is the collection of all the set partitions of $\{asp+1,\ldots,an\}$ into $n-sp$ sets of size $a$. We will also denote by $H^{(a^{n-sp})}_+$ the $\mathbb{F}S_{\{asp+1,\ldots,an\}}$-permutation module generated by $\Omega^{(a^{n-sp})}_+$ as a vector space. 
The map $$\psi:\Omega^{(a^n)}(R_{as})\longrightarrow \Omega^{(a^{sp})}(R_{as})\times \Omega^{(a^{n-sp})}_+$$
that associates to each $\omega\in \Omega^{(a^n)}(R_{as})$ the element $u_\omega\times v_\omega\in \Omega^{(a^{sp})}(R_{as})\times \Omega^{(a^{n-sp})}_+$ is a well defined bijection. %and is respected by the action of $N_{S_{an}}(R_{as})\cong N_{S_{asp}}(R_{as})\times S_{a(n-sp)}$.
This factorization of the linear basis of $H^{(a^n)}(R_{as})$ induces an isomorphism of vector spaces between $H^{(a^n)}(R_{as})$ and $H^{(a^{sp})}(R_{as})\boxtimes H^{(a^{n-sp})}_+$ that is compatible with the action of $N_{S_{an}}(R_{as})\cong N_{S_{asp}}(R_{as})\times S_{a(n-sp)}$. Therefore we have that 
$$H^{(a^n)}(R_{as})\cong H^{(a^{sp})}(R_{as})\boxtimes H^{(a^{n-sp})}_+$$ as $\mathbb{F}(N_{S_{an}}(R_{as}))$-modules. The proposition now follows after identifying 
$S_{\{asp+1,\ldots,an\}}$ with $S_{a(n-sp)}$ and $H^{(a^{n-sp})}_+$ with $H^{(a^{n-sp})}$.
\end{proof}

Lemma \ref{L1} allows us to restrict for the moment our attention to the study of the Brou\'e correspondent $H^{(a^{sp})}(R_{as})$ of $H^{(a^{sp})}$. In particular we will now give a precise description of its canonical basis $\Omega^{(a^{sp})}(R_{as})$ constitued by the set partitions fixed under the action of $R_{as}$. 
In order to do this we need to introduce a new important concept. 

Let $\delta=\{\delta_1,\delta_2,\ldots,\delta_s\}$ be a set partition of $\{1,2,\ldots,as\}$ into $s$ sets of size $a$ (namely $\delta\in\Omega^{(a^s)}$).
Let $A_1,A_2,\ldots,A_s$ be subsets of $\{1,2,\ldots,asp\}$ of size $a$ such that for each $i\in\{1,\ldots,s\}$ and $j\in\{1,\ldots,as\}$ we have 
\[
|A_i\cap \mathcal{O}_j|=
\begin{cases}
1\ \ \text{if}\ \ j\in \delta_i \\
0\ \ \text{if}\ \  j\notin \delta_i. 
\end{cases}
\]
In particular each set $A_i$ contains at most one element of a given orbit of $R_{as}$.
Consider now $\omega$ to be the element of $\Omega^{(a^{sp})}(R_{as})$ of the form
$$\omega=\{A_1,A_1\sigma,A_1\sigma^2,\ldots,A_1\sigma^{p-1},A_2,\ldots A_2\sigma^{p-1},\ldots\ldots,A_s,\ldots,A_s\sigma^{p-1}\},$$
where $\sigma=z_1z_2\cdots z_{as}$.
We will say that the set partition $\omega$ has \emph{type} $\delta$. Notice that from the type we can read how the orbits of $R_{as}$ are relatively distributed in the sets of the set partition $\omega$. 

In the following lemma we will show that the set partitions of $\Omega^{(a^{sp})}$ that are fixed by the action of $R_{as}$ are precisely the ones 
of the form described above.

\begin{lemma}\label{L2}
Let the set partition $\omega=\{\omega_1,\ldots,\omega_{sp}\}$ be an element of $\Omega^{(a^{sp})}$. Then $\omega$ is fixed by $R_{as}$ if and only if 
there exists a corresponding set partition $\delta=\{\delta_1,\ldots,\delta_s\}\in\Omega^{(a^s)}$ and $s$ sets $A_1,\ldots, A_s$ of $\omega$ such that 
\[
|A_i\cap \mathcal{O}_j|=
\begin{cases}
1\ \ \text{if}\ \ j\in \delta_i \\
0\ \ \text{if}\ \  j\notin \delta_i 
\end{cases}
\]
and $$\omega=\{A_1,A_1\sigma,A_1\sigma^2,\ldots,A_1\sigma^{p-1},A_2,\ldots A_2\sigma^{p-1},\ldots\ldots,A_s,\ldots,A_s\sigma^{p-1}\},$$
where $\sigma=z_1z_2\cdots z_{as}$.
\end{lemma}

\begin{proof}
Suppose that $\omega$ is fixed by $R_{as}=\left\langle\sigma\right\rangle$. Let $\mathcal{O}_j$ be an orbit of $R_{as}$ for some $j\in\{1,2,\ldots,as\}$ and let 
$\omega_{j_1},\omega_{j_2},\ldots,\omega_{j_l}$ be the sets of $\omega$ such that $\omega_{j_i}\cap \mathcal{O}_j\not=\emptyset$. Clearly $l\leq p$ because $|\mathcal{O}_j|=p$. Since $\omega\sigma=\omega$ we have that for all $x\in\{j_1,\ldots,j_l\}$ there exists $y\in\{j_1,\ldots,j_l\}$ such that $x\neq y$ and
$\omega_{j_x}\sigma=\omega_{j_y}$ and no $\omega_{j_i}$ is fixed by $\sigma$ because $a<p$. In particular we have that $R_{as}$ acts without fixed points on the set $\{\omega_{j_1},\omega_{j_2},\ldots,\omega_{j_l}\}$. Therefore there exists a number $k\geq 1$ such that $$p^k=|\{\omega_{j_1},\omega_{j_2},\ldots,\omega_{j_l}\}|=l\leq p.$$
This immediately implies that $l=p$ and therefore that $|\omega_{j_i}\cap \mathcal{O}_j|=1$ for all $i\in\{1,2,\ldots,l\}$. This argument holds for all the 
$R_{as}$ orbits $\mathcal{O}_1,\mathcal{O}_2,\ldots,\mathcal{O}_{as}$. Hence for all $x\in\{1,2,\ldots,sp\}$ the set $\omega_x$ of $\omega$ contains $a$ numbers no two of which are in the same $R_{as}$-orbit. 
Consider one of those sets, say $A_1$, of $\omega$. Define the correspondent set $\delta_1$ of size $a$ as follows: 
for all $i\in\{1,2,\ldots,as\}$, let $i\in\delta_1$ if and only if $|A_1\cap\mathcal{O}_i|=1$. Observe that since $\omega\sigma=\omega$ we have that 
$A_1,A_1\sigma,\ldots,A_1\sigma^{p-1}$ are $p$ distinct sets of $\omega$ such that for all $k\in\{0,1,\ldots,p-1\}$ we have that  
\[
|A_1\sigma^k\cap \mathcal{O}_j|=
\begin{cases}
1\ \ \text{if}\ \ j\in \delta_1 \\
0\ \ \text{if}\ \  j\notin \delta_1. 
\end{cases}
\]
We now repeat the above construction by considering a set $A_2$ of $\omega$ such that $A_2\not=A_1\sigma^k$ for any $k\in\{0,1,\ldots,p-1\}$ and defining the corresponding set $\delta_2$ exactly as above. After $s$ iterations of the process we obtain the claim, where the set partition $\delta\in \Omega^{(a^s)}$ corresponding to $\omega$ is $\delta=\{\delta_1,\delta_2,\ldots,\delta_s\}$.  

The converse is trivial since a set partition $\omega$ of the form described in the hypothesis is clearly fixed by the action of $\sigma$.
\end{proof}

From Lemma \ref{L2} we obtain that every $\omega\in\Omega^{(a^{sp})}(R_{as})$ is of a well defined type $\delta\in\Omega^{(a^s)}$. In the next lemma we will fix a $\delta\in\Omega^{(a^s)}$ and we will
calculate explicitly the number of set partitions of type $\delta$ in $\Omega^{(a^{sp})}(R_{as})$.

\begin{lemma}\label{numb}
For every given $\delta\in\Omega^{(a^{s})}$ there are $p^{(a-1)s}$ distinct set partitions in $\Omega^{(a^{sp})}(R_{as})$ of type $\delta$.
\end{lemma}
\begin{proof}
Define $\delta^\star\in\Omega^{(a^s)}$ by $$\delta^\star=\big\{\{1,1+s,\ldots,1+(a-1)s\},\{2,2+s,\ldots,2+(a-1)s\},\cdots,\{s,2s,\ldots,as\}\big\}.$$
By Lemma \ref{L2} we have that given any set partition $\omega=\{\omega_1,\ldots,\omega_{sp}\}\in \Omega^{(a^{sp})}(R_s)$ of type $\delta^\star$,
then each set $\omega_j$ contains exactly one element lying in $\{1,2,\ldots,sp\}$, the union of the first $s$ orbits $\mathcal{O}_1,\mathcal{O}_2,\ldots,\mathcal{O}_s$ of $R_{as}$.
Therefore, without loss of generality, we can relabel the indices of the sets of $\omega$ in order to have for all $j\in\{1,2,\ldots,sp\}$ $$\omega_j=\{j,x_j^1,x_j^2,\ldots,x_j^{a-1}\},$$
where $x_j^i$ is a number lying in the $R_{as}$-orbit of $j+isp$ for each $i\in\{1,2,\ldots,a-1\}$.
Notice that this implies that there are $p$ possible different choices for each $x_j^i$.
If we fix $j\in\{1,2,\ldots,sp\}$ such that $j$ is not divisible by $p$ then there exist unique natural numbers $t$ and $k$ in $\{0,1,\ldots,s-1\}$ and $\{1,2,\ldots,p-1\}$ respectively, such that $j=tp+k$.
Moreover, by definition of $\sigma$ it follows that $((t+1)p)\sigma^k=j$. Since $\omega\sigma^k=\omega$, we must have $\omega_{(t+1)p}\sigma^k=\omega_j$. Therefore for all $i\in\{1,2,\ldots,a-1\}$ we have that 
$$x_j^i=x_{(t+1)p}^i\sigma^{k}$$
Hence the set partition $\omega$ is uniquely determined by its sets $\omega_p,\omega_{2p},\ldots,\omega_{sp}$. This implies that there are exactly 
$p^{(a-1)s}$ different set partitions of type $\delta^\star$ in $\Omega^{(a^{sp})}(R_{as})$. It is an easy exercise to verify that, changing the labels, the argument above works for any other type $\delta$ in $\Omega^{(a^{s})}$. 
\end{proof}

Consider now the subgroup of $N_{S_{asp}}(R_{as})$ defined by $$C:=\left\langle z_1\right\rangle\times\left\langle z_2\right\rangle\times\ldots\times\left\langle z_{as}\right\rangle.$$
Notice that $C$ preserves the type of set partitions in its action on $\Omega^{(a^{sp})}(R_{as})$. Therefore we have that the sub-vectorspace $K_\delta$ of $H^{(a^{sp})}(R_{as})$ generated by all the fixed set partitions of type $\delta$ is an $\mathbb{F}C$-submodule of $H^{(a^{sp})}(R_{as})$ for any given $\delta$. Moreover, we deduce the following result: 

\begin{prop}\label{ZZZZ}
The following isomorphism of $\mathbb{F}C$-modules holds: 
$$H^{(a^{sp})}(R_{as})\big\downarrow_C\cong \bigoplus_{\delta\in \Omega^{(a^s)}}K_\delta.$$ 
\end{prop}
\begin{proof}
For any given $\delta\in\Omega^{(a^s)}$ denote by $B_\delta$ the subset of $\Omega^{(a^{sp})}(R_{as})$ containing all the set partitions of type $\delta$.
Clearly $H^{(a^{sp})}(R_{as})$ decomposes as a vector space into the direct sum of all the $K_\delta$ for $\delta\in\Omega^{(a^s)}$. Moreover we observe that 
the orbits of $C$ on $\{1,2,\ldots,asp\}$ are exactly the same as the orbits of $R_{as}$ and therefore if $\omega\in B_\delta$ then $\omega c\in B_\delta$ for any $c\in C$. 
This implies that $$H^{(a^{sp})}(R_{as})\downarrow_C \cong \bigoplus_{\delta\in \Omega^{(a^s)}}K_\delta$$ as $\mathbb{F}C$-modules, as desired.
\end{proof}

We will now define three $p$-subgroups of $S_{asp}$ that will play a central role in the next part of the section. 
For all $j\in\{1,2,\ldots,s\}$ denote by $\pi_{j}$ the $p$-element of $C$ given by $$\pi_j=\overline{z_j}=z_jz_{j+s}z_{j+2s}\cdots z_{j+(a-1)s}.$$
Let $E_s$ be the elementary abelian subgroup of $C$ of order $p^s$ defined by
 
$$E_s = \langle \pi_1 \rangle \times \cdots
\times \langle \pi_s \rangle.$$ 

%where the $z_j$ are the $p$-cycles defined at the start of this section. 
%For $i\in\{1,\ldots,sp\}$, $k\in\{0,1,\ldots,a-1\}$ and for $g\in S_{\{1,\ldots,sp\}}$ let $\overline{g}$ be the permutation of $S_{a}\wr S_{sp}$ such that 
%$(i+ksp)\overline{g}=ig+ksp$.
%Notice that if $1\leq i\leq s$ then $\overline{z_i}=\pi_i$.
Let $P_s$ be a Sylow $p$-subgroup
of $S_{\{1,\ldots, sp\}}$ with base group $\langle z_1, \ldots, z_s \rangle$, chosen 
so that $z_1z_2\cdots z_s$ is in its centre. (The existence of such Sylow $p$-subgroups follows from
the construction of Sylow $p$-subgroups of symmetric groups
as iterated wreath products in \cite[4.1.19 and 4.1.20]{JK}).

Let $Q_s$ be the group consisting of all permutations $\overline{g}$ 
where $g$ lies in $P_s$. For the reader's convenience we recall that for all $k\in\{0,1,\ldots,a-1\}$ and all $j\in\{1,2,\ldots,sp\}$, we have that 
$$(j+ksp)\overline{g}=(j)g+ksp.$$
In particular we observe that this implies that $\overline{g}=g_0g_1\cdots g_{a-1}$, where for all $k\in\{0,1,\ldots a-1\}$, $g_k$ is the element of $S_{asp}$ that fixes all the numbers outside $\{ksp+1,ksp+2,\ldots,(k+1)sp\}$ and such that $(j+ksp)g_k=(j)g+ksp$ for all $j\in\{1,2,\ldots,sp\}$. 
Notice that $Q_s$ has $E_s$ as normal base group by construction and clearly $R_{as}\triangleleft E_s\triangleleft C$ and $R_{as}\triangleleft Q_s$. 

We are now very close to deducing the indecomposability of $H^{(a^{sp})}(R_{as})$ as an $\mathbb{F}N_{S_{asp}}(R_{as})$-module. In order to prove this we need to observe an important structural property of the $\mathbb{F}C$-modules $K_\delta$ for all $\delta\in\Omega^{(a^s)}$.

\begin{prop}\label{PC}
For any $\delta\in\Omega^{(a^s)}$ there exists $g\in N_{S_{asp}}(R_{as})$ such that $$K_\delta\cong \mathbb{F}\uparrow_{E_s^g}^C$$
\end{prop}

\begin{proof}
As usual, define $\delta^\star\in\Omega^{(a^s)}$ by $$\delta^\star=\{\delta_1,\delta_2,\ldots,\delta_s\},$$ where $\delta_i=\{i,i+s,i+2s,\ldots,i+(a-1)s\}$ and let 
$\omega^\star$ be any fixed element of $B_{\delta^\star}$. Then by Lemma \ref{L2} we have that $$\omega^\star=\{A_1,A_1\sigma,\ldots,A_1\sigma^{p-1},A_2,\ldots A_2\sigma^{p-1},\ldots\ldots,A_s,\ldots,A_s\sigma^{p-1}\},$$
for some sets $A_1,A_2,\ldots,A_s$ such that 
\[
|A_i\cap \mathcal{O}_j|=
\begin{cases}
1\ \ \text{if}\ \ j\in \delta^\star_i\\
0\ \ \text{if}\ \  j\notin \delta^\star_i. 
\end{cases}
\]
This implies that we can equivalently rewrite $\omega^\star$ as $$\omega^\star=\{A_1,A_1\pi_1,\ldots,A_1\pi_1^{p-1},A_2,A_2\pi_2\ldots A_2\pi_2^{p-1},\ldots\ldots,A_s,\ldots,A_s\pi_s^{p-1}\}.$$
Therefore $\omega^\star$ is fixed by the action of $E_s$.
Moreover if we denote by $L$ the stabilizer in $C$ of $\omega^\star$ we have that as $\mathbb{F}C$-modules $$K_{\delta^\star}\cong \mathbb{F}\uparrow_L^C$$ since $C$ acts transitively on the elements of $B_{\delta^\star}$. 

Lemma \ref{numb} implies that $\dim_{\mathbb{F}}(K_{\delta^\star})=p^{s(a-1)}$, therefore by \cite[page 56]{Alperin} we have that 
$$p^{s(a-1)}=|C:L|\leq |C:E_s|=p^{s(a-1)}.$$
Hence $E_s=L$ and $K_{\delta^\star}\cong \mathbb{F}\uparrow_{E_s}^C$ as $\mathbb{F}C$-modules. Since $N_{S_{asp}}(R_{as})$ acts as the full symmetric group on the set $\{\mathcal{O}_1,\mathcal{O}_2,\ldots,\mathcal{O}_{as}\}$, we obtain that for any $\delta\in\Omega^{(a^s)}$ 
there exists $g\in N_{S_{asp}}(R_{as})$ such that any set partition of type $\delta$ in $\Omega^{(a^{sp})}(R_{as})$ is fixed by $E_s^g$. With an argument completely similar to the one used above we deduce that $$K_{\delta}\cong \mathbb{F}\uparrow_{E_s^g}^C.$$
\end{proof}

The following corollary of Proposition \ref{PC} will be extremely useful in the last part of the section.

\begin{cor}\label{C1}
Every indecomposable summand of $H^{(a^{sp})}(R_{as})$ has vertex containing $E_s$.
\end{cor}

\begin{proof}
Let $U$ be an indecomposable summand of $H^{(a^{sp})}(R_{as})$. By Lemma \ref{Lemma:perm} and Proposition \ref{PC} we observe that the restriction of $U$ to $C$ is isomorphic to 
a direct sum of indecomposable $p$-permutation $\mathbb{F}C$-modules with vertices conjugate in $N_{S_{asp}}(R_{as})$ to $E_s$. Therefore by the first part of Corollary \ref{BC3}
we obtain that $E_s$ is contained in a vertex of $U$.
\end{proof}
It is now possible to determine a vertex of $H^{(a^{sp})}(R_{as})$ as an $\mathbb{F}N_{S_{asp}}(R_{as})$-module.
\begin{prop}\label{P4}
The $\mathbb{F}N_{S_{asp}}(R_{as})$-module $H^{(a^{sp})}(R_{as})$ is indecomposable and has vertex $Q_s\in \mathrm{Syl}_p(S_a\wr S_{sp})$
\end{prop}
\begin{proof}
Let $\delta^\star$ be the set partition of $\Omega^{(a^s)}$ defined at the beginning of the proof of Proposition \ref{PC}.
Since $\omega\in\Omega^{(a^{sp})}(R_{as})$ is fixed by $E_s$ if and only if $\omega\in B_{\delta^\star}$ and since $E_s\triangleleft C$, we have that $$(H^{(a^{sp})}(R_{as}))(E_s)\downarrow_C = H^{(a^{sp})}(R_{as})\downarrow_C(E_s)=K_{\delta^\star}(E_s)\cong \mathbb{F}\uparrow_{E_s}^C,$$
as $\mathbb{F}C$-modules.
By Lemma \ref{Lemma:perm} we have that $(H^{(a^{sp})}(R_{as}))(E_s)\downarrow_C$ is indecomposable, hence also $(H^{(a^{sp})}(R_{as}))(E_s)$ is indecomposable. Therefore by Proposition \ref{pp1} there exists a unique summand of $H^{(a^{sp})}(R_{as})$ with vertex containing 
$E_s$, but this implies that $H^{(a^{sp})}(R_{as})$ is indecomposable by Corollary \ref{C1}.

Let $Q\leq_p N_{S_{asp}}(R_{as})$ be a vertex of $H^{(a^{sp})}(R_{as})$.
Consider $\omega^\star$ to be the set partition in $\Omega^{(a^{sp})}(R_{as})$ defined by $$\omega^\star=\{\omega_1,\omega_2,\ldots,\omega_{sp}\},$$
where $\omega_j=\{j,j+sp,j+2sp,\ldots,j+(a-1)sp\}$, for all $j\in\{1,2,\ldots,sp\}$. By construction we have that $Q_s$ fixes $\omega^\star$. Therefore a conjugate of $Q_s$ is a subgroup of $Q$. On the other hand by Corollary \ref{cor:Brauer} there exists $\omega\in\Omega^{(a^{sp})}(R_{as})$ such that $Q$ fixes $\omega$. Since the stabilizer of $\omega$ in $S_{asp}$ is isomorphic to $S_a\wr S_{sp}$, we deduce that $Q$ is isomorphic to a subgroup of a Sylow $p$-subgroup of $S_a\wr S_{sp}$. In particular this implies that $|Q|\leq |Q_s|$ and therefore we obtain that $Q_s$ is a vertex of $H^{(a^{sp})}(R_{as})$.
\end{proof}

\begin{cor}\label{CS2}
The Foulkes module $H^{(a^{sp})}$ has a unique indecomposable summand $U$ with vertex $Q_s\in \mathrm{Syl}_p(S_a\wr S_{sp})$. 
In particular $U$ is the Scott module $\mathrm{Sc}(S_{asp},S_a\wr S_{sp})$ and we have that $$H^{(a^{sp})}(Q_s)=U(Q_s)\cong P_{\mathbb{F}}$$ as $\mathbb{F}(N_{S_{asp}}(Q_s)/Q_s)$-modules. Moreover, any other indecomposable summand of $H^{(a^{sp})}$ has vertex conjugate to a subgroup of $Q_s$. 
\end{cor}

\begin{proof}
Since $H^{(a^{sp})}$ is isomorphic to the permutation module induced by the action of $S_{asp}$ on the cosets of $S_a\wr S_{sp}$, it is clear that any vertex of an indecomposable summand is contained in a Sylow $p$-subgroup of $S_a\wr S_{sp}$ and therefore is conjugate to a subgroup of $Q_s$.
From Proposition \ref{P4} we have that $H^{(a^{sp})}(R_{as})$ is indecomposable. Therefore by Proposition \ref{pp1} we deduce that exists a unique indecomposable summand $U$ of $H^{(a^{sp})}$ such that $R_{as}$ is contained in a vertex of $U$. We know that $R_{as}\leq Q_s$ and certainly the Scott module 
$\mathrm{Sc}(S_{asp},S_{a}\wr S_{sp})=\mathrm{Sc}(S_{asp},Q_s)$ is an indecomposable summand of $H^{(a^{sp})}$, with vertex $Q_s$. Therefore $U=\mathrm{Sc}(S_{asp},Q_s)$ and 
by Theorem \ref{scott} module we have that $$H^{(a^{sp})}(Q_s)=U(Q_s)\cong P_{\mathbb{F}},$$ as $\mathbb{F}(N_{S_{asp}}(Q_s)/Q_s)$-modules.
\end{proof}

In order to prove Theorem \ref{T1} we need the following technical lemma, which is the analogous of \cite[Lemma 4.6]{GW}. 
Denote by $D_s$ the group $C\cap N_{S_{asp}}(Q_s)$.

\begin{lemma}\label{LemmaMark}
Let $p$ be a prime and let $a$ and $s$ be natural numbers such that $a<p$. Then the unique Sylow $p$-subgroup of $N_{S_{asp}}(Q_s)$ is the subgroup 
$\langle D_s, Q_s\rangle$ of $N_{S_{asp}}(R_{as})$.
\end{lemma}

\begin{proof}

Keeping the notation introduced after Proposition \ref{ZZZZ}, for $j\in\{1,2,\ldots,as\}$ let $\mathcal{O}_j=\{(j-1)p+1,\ldots,jp\}$ and for $k\in\{1,\ldots,s\}$ let $$X_k=\bigcup_{l=0}^{a-1}\mathcal{O}_{ls+k}.$$
Since $Q_s$ normalizes $E_s$, it permutes the sets $X_1,\ldots,X_s$ as blocks for its action. Moreover given $x\in N_{S_{asp}}(Q_s)$ we have that 
$\pi_j^x\in\langle\pi_1,\ldots,\pi_s\rangle$ for all $j\in\{1,\ldots,s\}$. Therefore also $N_{S_{asp}}(Q_s)$ permutes as blocks for its action the sets 
$X_1,\ldots, X_s$. 

Let $g$ be a $p$-element of $N_{S_{asp}}(Q_s)$. The group
$\langle Q_s, g \rangle$ permutes the sets in $X:=\{X_1,\ldots,X_s\}$ as blocks for its action. Let
$$ \pi : \langle Q_s, g \rangle \rightarrow S_X $$
be the corresponding group homomorphism. 
By construction $Q_s$ acts on the sets $X_1, \ldots, X_s$
as a Sylow $p$-subgroup of $S_{\{ X_1, \ldots, X_s \}}$; hence
$Q_s\pi$ is a Sylow $p$-subgroup of $S_X$.
Therefore, since $\langle Q_s, g\rangle$
is a $p$-group, there exists $\tilde{g} \in Q_s$ such that $g \pi = \tilde{g} \pi$. Let $y =
g\tilde{g}^{-1}$. Since $y$ acts trivially on the sets in $X$, we may write
\[ y = g_1 \ldots g_s \]
where $g_j \in S_{X_j}$ for each $j$.
The $p$-group $\langle Q_s, y \rangle$ has as a subgroup $\langle \pi_j, y 
\rangle$. The permutation group induced by the subgroup on $X_j$, namely $\langle \pi_j, g_j\rangle$,
is a $p$-group acting on a set of size $ap$. Since $p>a$, the unique Sylow
$p$-subgroup of $S_{X_j}$ 
containing $\pi_j$
is $\langle z_{j}, z_{j+s}, \ldots, z_{j+(a-1)s} \rangle$. Hence $g_j \in \langle z_{j}, z_{j+s},\ldots, z_{j+(a-1)s} \rangle$ for
each $j\in \{1,\ldots, s\}$. Therefore $y \in \langle z_1, z_2 \ldots, z_{as}\rangle=C$. We also know that $y\in \langle Q_s,g\rangle\leq N_{S_{asp}}(Q_s)$. Therefore $y\in D_s$, 
and 
since $\tilde{g} \in Q_s$, it follows that $g \in \langle D_s, Q_s \rangle$, as required.
Conversely, the subgroup $\langle D_s, Q_s\rangle$ is contained in $N_{S_{asp}}(Q_s)$ because both $D_s$ and $Q_s$ are.
It follows that $\langle D_s, Q_s \rangle$ is the unique Sylow $p$-subgroup of $N_{S_{asp}}(Q_s)$.
\end{proof}

We are now ready to prove Theorem \ref{T1}

\begin{proof}[Proof of Theorem~\ref{T1}]
To simplify the notation we denote $N_{S_{asp}}(R_{as})$ by $K_s$.
Let $U$ be an indecomposable summand of $H^{(a^n)}$ with vertex $Q$. Let $\ell\in \{1,\ldots,\lfloor\frac{an}{p}\rfloor\}\cap\mathbb{N}$ be maximal with respect to the property that $R_\ell$ is a subgroup of (a conjugate of) the vertex $Q$. 
The Brou\'e correspondent $U(R_\ell)$ is a non-zero direct summand of $H^{(a^n)}(R_\ell)$ by Theorem \ref{BT2}. Therefore we deduce by Lemma \ref{L1} that there exist a natural number $s$ such that $\ell=as$ and $Z$ a non-zero summand of $H^{(a^{n-sp})}$ such that 
$$U(R_{as})\cong H^{(a^{sp})}(R_{as})\boxtimes Z$$ 
as $\mathbb{F}(K_s\times S_{a(n-sp)})$-modules. Since $R_{as}$ is normal in $Q_s$, it follows from Lemmas \ref{LDB} and \ref{LDB2} that there is an isomorphism of $\mathbb{F}(N_{K_s}(Q_s)\times S_{a(n-sp)})$-modules $$U(Q_s)\cong (U(R_{as}))(Q_s)\cong (H^{(a^{sp})}(R_{as}))(Q_s)\boxtimes Z.$$
By Proposition \ref{P4} we deduce that $H^{(a^{sp})}(Q_s)\boxtimes Z\neq 0$. Hence we have that $Q_s\leq Q$. 

Let $\mathcal{B}$ be a $p$-permutation basis for the $\mathbb{F}K_s$-module $H^{(a^{sp})}(R_{as})$ with respect to a Sylow $p$-subgroup of $K_s$ containing $Q_s$. It follows from Corollary \ref{cor:Brauer} and Lemma \ref{LemmaMark} that $\mathcal{C}=\mathcal{B}^{Q_s}$ is a $p$-permutation basis for the $\mathbb{F}N_{K_s}(Q_s)$-module $(H^{(a^{sp})}(R_{as}))(Q_s)$ 
with respect to the unique Sylow $p$-subgroup $P:=\left\langle D_s, Q_s\right\rangle$ of $N_{K_s}(Q_s)$. Let $\mathcal{C}'$ be a $p$-permutation basis for $Z$ with respect to $P'$, a Sylow $p$-subgroup of $S_{\{asp+1,\ldots,an\}}\cong S_{a(n-sp)}$.
Hence $$\mathcal{C} \boxtimes \mathcal{C}' = \{ v \otimes v' : v \in \mathcal{C},
v' \in \mathcal{C}'\}$$
is a $p$-permutation basis for $(H^{(a^{sp})}(R_{as}))(Q_s) \boxtimes Z$ with
respect to the Sylow $p$-subgroup 
$P\times P'$ of $N_{K_{s}}(Q_s) \times S_{a(n-sp)}$.

Suppose, for a contradiction, that $Q$ strictly contains $Q_s$. 
Since~$Q$ is a $p$-group there exists a $p$-element
$g \in N_{Q}(Q_s)$ such that
$g \not\in Q_s$. Notice that $Q_s$ has orbits of length at least $p$
on $\{1,\ldots, asp\}$ and fixes $\{asp+1,\ldots, an\}$. Since $g$
permutes these orbits as blocks for its action, we may factorize
$g$ as $g = hh^+$ where $h\in N_{S_{asp}}(Q_s)$ and %$h \in S_{\{1,\ldots,rp\}}$ and 
$h^+ \in S_{a(n-sp)}$. %\{rp+1,\ldots, 2m+k\}}$.
By Lemma \ref{LemmaMark} we have that $\langle Q_s, h \rangle \le N_K(Q_s)$. 

Corollary~\ref{cor:Brauer} now implies that
$(\mathcal{C} \boxtimes \mathcal{C}')^{\langle Q_s, g\rangle}
\not=\varnothing$. Let \hbox{$v \otimes v' \in \mathcal{C} \boxtimes \mathcal{C}'$}
be such that $(v \otimes v')g = v \otimes v'$.
Then $v \in \mathcal{B}^{\langle Q_s, h \rangle}$.
But $Q_s$ is a vertex
of $H^{(a^{sp})}(R_{as})$, %by Proposition~\ref{prop:Qvertex}, 
so it follows from
Corollary~\ref{cor:Brauer} that $h \in Q_s$. Hence $h'$ is a
non-identity element of~$Q$. By taking an appropriate power
of $h'$ we find that~$Q$ contains a product of one or more $p$-cycles
with support contained in \hbox{$\{asp+1, \ldots, an\}$}. This contradicts
our assumption that $l=as$ was maximal such that $R_{as}$ is contained
in a vertex of $U$.
Therefore $U$ has vertex $Q_s$. 

We saw above that there is an isomorphism $$U(Q_s) \cong (H^{(a^{sp})}(R_{as}))(Q_s)
\boxtimes Z$$ of $\mathbb{F}(N_{K_s}(Q_s) \times S_{a(n-sp)})$-modules.
This identifies $U(Q_s)$ as a vector space on which $N_{S_{an}}(Q_s)
=  N_{S_{asp}}(Q_s) \times S_{a(n-sp)}$ 
acts. It is clear from the isomorphism in Lemma \ref{L1}
that $N_{S_{asp}}(Q_s)$ acts on the first tensor factor and $S_{a(n-sp)}$
acts on the second. Hence the action of $N_{K_s}(Q_s)$ on $(H^{(a^{sp})}(R_{as}))(Q_s)$
extends to an action of $N_{S_{asp}}(Q_s)$ on $(H^{(a^{sp})}(R_{as}))(Q_s)$
and we obtain a tensor factorization $V \boxtimes Z$ of $U(Q_s)$
as a $N_{S_{asp}}(Q_s) \times S_{a(n-sp)}$-module. %an $N_{S_{2m+k}}(Q_\ell)$-module.
An outer tensor product of modules is projective if and only if
both factors are projective,
so by Theorem~\ref{BC1} and Corollary \ref{CS2},
$V$ 
%$M_\ell(Q_\ell)$ 
is isomorphic to the projective cover of the trivial $\mathbb{F}(N_{S_{asp}}(Q_s)/Q_s)$-module,
$Z$ is a projective $\mathbb{F}S_{a(n-sp)}$-module, and
$U(Q_s)$ is the Green correspondent of~$U$.
\end{proof}

Notice that Theorem \ref{T1} implies that the only indecomposable summands of the Foulkes module $H^{(a^n)}$ that are Young modules are the projective ones,
because no non-projective Young module can possibly have vertex of the form $Q_s\in \mathrm{Syl}_p(S_a\wr S_{sp})$.

Another consequence of Theorem \ref{T1} is that the modular version of Foulkes' Conjecture is false. In fact if we consider $a<p\leq n$ then 
$H^{(n^a)}$ cannot be a direct summand of $H^{(a^n)}$ since there exists a non-projective Young module $Y$ in the direct sum decomposition of $H^{(n^a)}$, namely the Scott module $\mathrm{Sc}(S_{an}, S_n\wr S_a)$.
As explained above such $Y$ cannot appear as a direct summand of $H^{(a^n)}$. 

It is important to underline that for a given natural number $a<p$, Theorem \ref{T1} tells us that an indecomposable summand $U$ of $H^{(a^n)}$ has vertex 
$Q_s$ for some $s\in\{1,\ldots,\lfloor\frac{n}{p}\rfloor\}\cap\mathbb{N}$ but it does not guarantee that for every $s\in\{1,\ldots,\lfloor\frac{n}{p}\rfloor\}\cap\mathbb{N}$ there exists an indecomposable summand with vertex $Q_s$. We believe that such situation occurs. It will be material of another article to prove this property for the modules $H^{(2^n)}$ for any odd prime $p$ and any $n\in\mathbb{N}$.

\section{Corollaries on the decomposition matrix}

In this section we will give upper bounds to the entries of some columns of the decomposition matrix of $\mathbb{F}_pS_{an}$ when $a<p$.
In particular we will prove Theorem \ref{TTT1}. 

For the rest of the section let $\mathbb{F}_p$ be the finite field of size $p$ and let $a,w$ be natural numbers such that $w<a<p$. Let $B:=B(\gamma,w)$ be a block of the group algebra $\mathbb{F}_pS_{an}$ such that $\mathcal{F}(\gamma)\neq\emptyset$, where $\mathcal{F}(\gamma)$ is the set defined in the introduction, before the statement of Theorem \ref{TTT1}. Denote by $D$ a defect group of $B$.
For every $p$-regular partition $\nu$ of $an$, we will denote by $P^\nu$ the projective cover of the simple $\mathbb{F}_pS_{an}$-module $D^\nu$.

\begin{prop}
The block component of $H^{(a^n)}$ for the block $B$ is projective.
\end{prop}
\begin{proof}
Let $U$ be an indecomposable summand of $H^{(a^n)}$ lying in the block $B$. Suppose that $U$ is non-projective. Then by Theorem \ref{T1} there exists $t\in\mathbb{N}$ non-zero such that $Q_t\in\mathrm{Syl}_p(S_a\wr S_{tp})$ is a vertex of $U$. By \cite[Theorem 5, page 97]{Alperin} we deduce that a conjugate of 
$Q_t$ is contained in $D$. Moreover, by Theorem \ref{Thm:SymDefectGroups} we have that $D$ is conjugate to a Sylow $p$-subgroup of $S_{wp}$. This leads to a contradiction, because $$|\mathrm{supp}(Q_t)|=atp\geq ap>wp=|\mathrm{supp}(D)|.$$
\end{proof}

In order to prove Theorem \ref{TTT1}, we will need Scott's lifting theorem \cite[Theorem 3.11.3]{Benson}.
For the reader convenience we state it below. We will denote by $\mathbb{Z}_p$ the ring of $p$-adic integers.

\begin{teo}\label{scotty}
If $U$ is a %an indecomposable 
direct summand of a permutation $\mathbb{F}_pG$-module $M$
then there is a %an indecomposable 
$\mathbb{Z}_p G$-module
$U_{\mathbb{Z}_p}$, unique up to isomorphism,
such that $U_{\mathbb{Z}_p}$ is a direct summand of $M_{\mathbb{Z}_p}$
and $U_{\mathbb{Z}_p} \otimes_{\mathbb{Z}_p} \mathbb{F}_p \cong U$.
\end{teo}

Let $P^\nu_{\mathbb{Z}_p}$ be the $\mathbb{Z}_p$-free
$\mathbb{Z}_pS_{an}$-module whose reduction modulo $p$ is
$P^\nu$. By Brauer reciprocity (see for instance %\cite[\S2]{Navarro} or 
\cite[\S 15.4]{Serre}), the ordinary character of~$P^\nu_{\mathbb{Z}_p}$~is 
\[ \tag{$\star$} \psi^\nu = \sum_\mu d_{\mu \nu} \chi^\mu. \]
It is well known that 
if $d_{\mu \nu} \not=0$ then $\nu$ dominates~$\mu$. Hence
the sum may be taken over those partitions $\mu$ dominated by~$\nu$.

We are now ready to prove Theorem \ref{TTT1}.

\begin{proof}[Proof of Theorem~\ref{TTT1}]
Since $\mathcal{F}(\gamma)\neq\emptyset$ the block component $W$ of $H^{(a^n)}$ for $B$ is non zero. 
Let $\zeta_1,\ldots,\zeta_s$ be the $p$-regular partitions of $an$ such that $$W=P^{\zeta_1}\oplus P^{\zeta_2}\oplus\cdots\oplus P^{\zeta_s}.$$
From the definition of $\mathcal{F}(\gamma)$ and by Theorem \ref{scotty}, it follows that
the ordinary character of $W$
is 
$$\psi^{\zeta_1} + \cdots + \psi^{\zeta_s}=\sum_{\mu\in\mathcal{F}(\gamma)}\big(\sum_{i=1}^s d_{\mu\zeta_i}\big)\chi^\mu.$$
By hypothesis $\lambda$ is a maximal partition in the dominance order
on $\mathcal{F}(\gamma)$, 
and by~($\star$) each $\psi^{\zeta_j}$ is a sum of ordinary irreducible
characters $\chi^\mu$ for partitions $\mu$ dominated by $\zeta_j$. 
Therefore one of the partitions $\zeta_j$ must equal $\lambda$, 
as required.

Therefore $P^\lambda$ is a direct summand of $H^{(a^n)}$ and $\psi^\lambda$ is a summand of the Foulkes character $\phi^{(a^n)}$. Hence
$$d_{\mu\lambda}=\left\langle\psi^{\lambda},\chi^\mu\middle\rangle\leq \middle\langle\phi^{(a^n)},\chi^\mu\right\rangle,$$
for all $\mu\vdash an$.
In particular if $\mu\notin \mathcal{F}(\gamma)$ then $d_{\mu\lambda}=0$.
\end{proof}

As already mentioned in the introduction, Theorem \ref{TTT1} allows us to recover new information on the decomposition numbers via the study of the ordinary Foulkes character $\phi^{(a^n)}$.
An example of this possibility is the following result.

\begin{cor}\label{x}
Let $\lambda$ be a $p$-regular partition of $na$. Denote by $\gamma$ the $p$-core of $\lambda$. If $\lambda$ is maximal in $\mathcal{F}(\gamma)$, then $[S^\mu:D^\lambda]=0$ for all $\mu\vdash na$ such that $\mu$ has more than $n$ parts. 
\end{cor}
\begin{proof}
It is a well known fact (see for instance \cite[Proposition 2.7]{Giannelli}) that if $\mu$ has more than $n$ parts then $$\left\langle\phi^{(a^n)},\chi^\mu\right\rangle=0.$$
The statement now follows from Theorem \ref{TTT1}.
\end{proof}

We conclude with an explicit example.

\begin{esem}\label{E:last}
Let $a=4$, $n=p=5$ and let $\lambda=(18,2)$ be a weight-$3$ partition of $20$. The $5$-core of $\lambda$ is $\gamma=(3,2)$ and the multiplicity of $\chi^\lambda$ as an irreducible constituent of $\phi^{(4^5)}$ is $1$, by \cite[Theorem 5.4.34]{JK}. Therefore $\lambda\in\mathcal{F}(\gamma)$ and it is clearly maximal under the dominance order. By Corollary \ref{x} we obtain a number of non-trivial zeros in the column labelled by $\lambda$ of the decomposition matrix of $S_{20}$ in characteristic $5$. For instance, all the partitions $\mu$ obtained from $(3,2,1^5)$ by adding two $5$-hooks lie in $\mathcal{F}(\gamma)$ and are such that $[S^\mu:D^\lambda]=0$.
\end{esem}

\section*{Acknowledgments}
I would like to thank my PhD supervisor Dr. Mark Wildon for his helpful advice which guided me throughout this work.

\end{document}